\documentclass{amsart}

\newcommand{\C}{\mathbb{C}}
\newcommand{\Z}{\mathbb{Z}}

\newcommand{\N}{\mathbb{N}}
\newcommand{\D}{\mathbb{D}}
\newcommand{\T}{\mathbb{T}}

\theoremstyle{plain}
\newtheorem{prop}{Proposition}[section]
\newtheorem{thm}[prop]{Theorem}
\newtheorem{lemma}[prop]{Lemma}
\newtheorem{cor}[prop]{Corollary}

\theoremstyle{definition}

\theoremstyle{remark}

\numberwithin{equation}{section}

\begin{document}

\title[Generalised Harmonic functions]{A complex Lie algebra of rotationally symmetric operators and their harmonics}

\date{\today}

\author{Markus Klintborg}

\address{Mathematics, Faculty of Textiles, Engineering and Business, Department of Engineering, Bor\aa s University, Sweden}

\email{markus.klintborg@hb.se}

\address{Mathematics, Faculty of Science, Centre for Mathematical Sciences, 
Lund University, Sweden}

\email{markus.klintborg@math.lu.se}

\begin{abstract}
We describe the solutions to a family of rotationally symmetric second order partial differential equations in the complex plane that arises from a four-dimensional complex Lie algebra whose spanning set generates the algebra from which such generalised harmonic functions derive. We show that every one of these solutions have a canonical series representation and retrieve those obtained in the case of Laplace and Helmholtz equation. These sums are given in confluent hypergeometric terms that asymptotically correspond to the complex exponential function.

\end{abstract}

\subjclass[2010]{Primary: 31A05; Secondary: 33C15, 35J15}

\keywords{Harmonic function, Power series, 
Confluent hyper\-geometric function, Bessel function}

\maketitle

\setcounter{section}{-1}

\section{Introduction}

It is a classical result of harmonic function theory that any harmonic function can be represented as an infinite converging sum of homogeneous polynomials within a sufficiently small neighbourhood of a point in its domain. Later extensions have also shown that these classical representations may be retrieved from a more general setting. We aim to build on these later developments and derive the homogeneous expansions for the solutions to a family of second order differential operators in the complex plane that retrieve those obtained in the case of Laplace and Helmholtz. 

The collection $\Omega$ with which we are concerned is given by the family of second order partial differential operators
\begin{equation}\label{pqrdiffoporder0}
M_{s,t,r;z}=\partial\bar\partial 
-s z\partial-t \bar{z} \bar\partial-r, \quad z \in \C,
\end{equation} 
for complex numbers $s,t,r \in \C$, where $\partial$ and $\bar{\partial}$ are the usual complex derivatives. The collection of such operators contains those commonly associated with Laplace and Helmholtz, and maintains the rotational invariance held by these two. In fact, these operators form the simplest conceivable such symmetric extension that admit both of these two classical cases. Indeed, the middle two terms that appear in \eqref{pqrdiffoporder0} are the only single first order differential operators of smallest degree that are rotationally invariant, while the remaining two terms remain familiarly so.

The last criteria can be said to constitute a least constraint for any viable framework that extends beyond these two classical cases, at least in so far as we expect to be able to decompose the problem into its constituent parts. In very loose terms, rotational invariance allows us to replace the higher-dimensional problem with a countable number of local ordinary problems, each considered with respect to the homogeneous parts separately. The local solutions may then be patched up to form a global solution, in the hope that the resulting sum converges in an appropriately chosen space of functions, commonly taken to be $C^{\infty}(U)$ for some open neighbourhood $U$ about the origin. It was thus raised to an axiom in \cite{GHFS} to account for those operators that have appeared in connection with generalised harmonic functions, the analysis of which crucially relies on their invariance under rotations. 

Formally speaking, the operators in \eqref{pqrdiffoporder0} can be taken as elements or generators of the $\C$-subalgebra $\mathfrak{H}_2=\langle \partial \bar \partial, z \partial, \bar z \bar \partial \rangle$ that is strictly contained in the algebra $\mathfrak{R}_2=\langle \partial \bar \partial, z \partial, \bar z \bar \partial, z \bar z \rangle$ of all operators in the second Weyl-algebra $A_2=A_2(\C)$ that commute with rotations. While weaker than that generated by the Laplace operator, which also remains invariant under displacements, this structure has been shown to be just rigid enough to accommodate the approach alluded to in the previous paragraph. This last claim rests on the assertion \cite[Cor. 4.4]{GHFS} that $D \in \mathfrak{R}_2$ commutes with rotations if and only if there is a sequence of ordinary differential operators $\{T_{m,D}\}_{m \in \Z}$ such that 
\begin{equation}
\label{Daction}
Df_m(|z|^2)\xi_m(z)=\xi_m(z)T_{m,D}f_m(|z|^2),
\end{equation} 
holds for all $m \in \Z$, where $\xi_m(z)=z^{m}$ for $m \geq 0$ and $\xi_m(z)=\bar z^{|m|}$ for $m < 0$. Provided that we can solve for each of the individual terms, a reasonable candidate for the solution to the homogeneous problem $M_{s,t,r}u=0$ corresponding to \eqref{pqrdiffoporder0} may then be taken to be 
\begin{equation*}
u(z) \sim \sum_{m=-\infty}^{\infty} f_m(|z|^2) \xi_m(z).
\end{equation*}
The identity in \eqref{Daction} therefore allows us to reduce the two-dimensional problem to an ordinary single-dimensional problem, and ensures that the span of each monomial term over the relevant ring or space of functions remains invariant under the action of any operator $D \in \mathfrak{R}_2$. 

The elements $T_{m,D}$ of the sequence associated with the operator $D \in \mathfrak{R}_2 \subset A_2$ are thus univariate operators in the first Weyl-algebra $A_1$ generated by $x$ and $d/dx$. This relation between the two ring theoretic frameworks $A_1$ and $A_2$ can also be understood in terms of mappings between the two, with the intention of specifying precisely which of the operators in \eqref{pqrdiffoporder0} that are mapped to the same points in $A_1$ for a given $m \in \Z$, as described in the first part of this text. This partitioning of the class in \eqref{pqrdiffoporder0} has its roots in earlier work on the subject, where it was realised that some parameter realms were easier to treat than others for certain families of operators in $\mathfrak{R}_2$, and for which a closer such analysis was required in order to treat the full parameter range. Examples include the setup in \cite{K}. And while the specific family considered in this text is more forgiving in such regards, the procedure may well be adapted to more general settings, and has proven to be a good starting point in the harmonic analysis for operators in $\mathfrak{H}_2$.  

Another distinctive aspect of $\Omega$ is that its closure $\bar \Omega$ under the vector space operations makes a four-dimensional complex Lie-algebra under the typical bi\-linear product. The defining relations under this operation can be neatly summed up as
\begin{equation*}
D_1D_2-D_2D_1=\gamma \partial\bar \partial,
\end{equation*}
for any given $D_1, D_2 \in \bar \Omega$ of the form
\begin{equation*} 
D_1=a_{1}+a_{2}z\partial+a_{3}\bar z \bar \partial + a_{4}\partial \bar \partial, \quad D_2=b_{1}+b_{2}z\partial+b_{3}\bar z \bar \partial + b_{4}\partial \bar \partial,
\end{equation*}
where
\begin{equation*}
\label{commutatormultiple}
\gamma=a_{4}(b_{2}+b_{3})-b_{4}(a_{2}+a_{3}).
\end{equation*}
The computations involved are straightforward, noting that the first order terms commute and that $[\partial \bar \partial, z \partial] \ = [\partial \bar \partial, \bar z \bar \partial]= \partial \bar \partial$. 

The last can be compared to some of the earlier well-studied cases. For example, and in the guiding case\footnote{A case that we will return to on several occasions, for comparative means and in order to keep repetition to a minimum.} of 
\begin{equation}
\label{order1weightoperators}
L_{c_1,c_2,c_3,c_4}=c_1(1-\lvert z\rvert^2)\partial_z\bar\partial_z 
+c_2z\partial_z+c_3\bar z\bar\partial_z+c_4 \in \mathfrak{H}_2,
\end{equation} 
that was treated in \cite{kopqseries}, the commutator of any two such operators remains as before, but now with the problem that the operator $\partial \bar \partial$ of Laplace is nowhere contained, and so fails in this regard without a suitable extension.\footnote{An infinite extension in fact, and one that contains the polynomial algebra generated by $\partial \bar \partial$. This can be seen by first noting that $[\partial \bar \partial,z\partial]=\partial \bar \partial$, and then making use of the standard identity $[x^n,y]=\sum_{j=0}^{n-1} x^j [x,y] x^{n-j-1}$ to show that $\big[\partial^k \bar \partial^k,[\partial \bar \partial,(1-|z|^2)\partial \bar \partial]\big]=2k\partial^{k+1}\bar \partial^{k+1}$.} 

Proceeding to the analytic and main parts of this text, we shall refer to a function $u$ defined on the open ball $B_{\rho}$ of radius $\rho > 0$ centred at zero as a solution or generalised harmonic function on $B_{\rho}$ if $u$ is two times continuously differentiable on $B_{\rho}$ and satisfies
\begin{equation}
\label{pqrdiffoporder0equation}
M_{s,t,r}u=0\quad \text{in}\ B_{\rho}.
\end{equation} 
Implicit in this last description are the three parameters $s,t,r \in \C$ with respect to which the function $u$ satisfies \eqref{pqrdiffoporder0equation}, and we have chosen to drop the adjective ''$(s,t,r)-$'' that sometimes accompany or forego the term ''function'' in reference to those generalised harmonic functions that satisfy an equation similar to \eqref{pqrdiffoporder0equation}. And while it makes sense to divide between certain cases of the given parameter range initially, our end result will be indifferent to such a preliminary distinction. We also note that the classical case of Helmholtz is retrieved by setting the first two para\-meters $s$ and $t$ to zero in \eqref{pqrdiffoporder0equation}, while Laplace equation is retrieved when all three parameters are ignored. 

To state our primary result, we introduce the family of functions 
\begin{equation}\label{pqrorderzerobasicfunctionsintro}
\mathcal{P}(\alpha,\beta|\gamma;z)=\sum_{m=0}^{\infty}\frac{(\alpha,\beta)_{m}}{(\gamma)_{m}}\frac{z^m}{m!}, \quad z \in \C,
\end{equation}  
for suitably chosen parameters $\alpha,\beta,\gamma \in \C$. The symbol $( x,y )_n$ in \eqref{pqrorderzerobasicfunctionsintro} denotes the generalised Pochhammer symbol  
\begin{equation}
\label{ksymbol}
(x,y)_{n}=x(x+y)(x+2y)\ldots(x+ny-y), \quad (x,y)_{0}=1,
\end{equation}
and agrees with the usual Pochhammer symbol $(x)_n$ when $y$ is set to unity.\footnote{When $y=k$ is taken to be a natural number, these symbols are also referred to as "Poch\-hammer k-symbols" \cite{diaz}, commonly denoted $(x)_{n,k}=(x,k)_n$.} 

We will then show that any generalised harmonic $u$ is smooth on $B_{\rho}$ and that it can be expanded as an absolutely converging sum of homogeneous terms in the form of 
\begin{align}\label{pqrorderzeroseriesintro}
u(z)&=\sum_{m=0}^{\infty}\frac{\partial^mu(0)}{m!}\mathcal{P}(r+sm,s+t|m+1;|z|^2)z^m \\ &+\sum_{m=1}^{\infty} \frac{\bar \partial^mu(0)}{m!}\mathcal{P}(r+tm,t+s|m+1;|z|^2) \bar z^m, \quad z \in B_{\rho} \notag.
\end{align}
In this connection, we shall initially divide between two cases $s \neq -t$ and $s=-t$. The first of these two, which we will refer to as the case of Kummer, involves the confluent hypergeometric function 
\begin{equation*}
\Phi(a,b,z)=\sum_{m=0}^{\infty}\frac{(a)_m}{(b)_m}\frac{z^m}{m!}, \quad z \in \C.
\end{equation*}
It is an entire function for $b \in \C  \setminus \{0,-1,-2,\ldots\}$ and perhaps the more familiar of the family of hypergeometric functions following that of Gauss. The latter case concerns Bessel's modified function
\begin{equation*}
I_m(z)=i^{-m}J_{m}(iz), \quad z \in \C,
\end{equation*}
where $J_{m}$ is the common Bessel function parametrized by $m \in \N$. We will point to these representational forms on numerous occasions, and show that the functions in \eqref{pqrorderzerobasicfunctionsintro} can be expressed in terms of Kummer when $s\neq -t$, while being those of Bessel in case of the latter.      

We will also show that 
\begin{equation*}
\lim\limits_{m \to \infty} \mathcal{P}(r+sm,s+t|m+1;z)=e^{sz},
\end{equation*}
where convergence is to be taken in uniform terms on compact subsets.  
 
This work is one in a series of reports on generalised harmonic functions with special emphasis on its series representations. Its intent is to contribute towards the growing number of examples that suggest that a coherent theory for such functions may be within reach under certain restrictions in regards to the symmetries that underlie their harmonic analysis, and to point out their intricate relation to those functions of a special kind. There is also good reason to believe that the content of the solutions that arise from the family $\Omega$ in \eqref{pqrdiffoporder0} is indicative of a more general phenomena. The last presumption rests in part on the findings cited throughout this text, and on numerous related studies \cite{ABC,Behm,BH,CW,Perala,Garabedian, Geller,LC,LW,LWX,O,OLipschitz,OWUHP,W}. We also suggest that the structural foundation natural to such functions is the one adopted above, in relation to which the given examples are of significant importance. And while earlier perspectives can be adopted, we have here taken as our starting point the work of Olofsson and Wittsten \cite{OW}, Olofsson \cite{O14}, and the even more recent studies referred to above.

\section{Preliminaries} 
 
The partial differential operators in \eqref{pqrdiffoporder0} are rotationally invariant. In contrast to the case of Laplace and that of Helmholtz however, it is in general not true that the translate $u(z_0+z)$ of a solution $u$ to \eqref{pqrdiffoporder0equation} is again a solution to this equation. This difference is explained by the middle two terms $z \partial$ and $\bar{z} \bar{\partial}$, and accounts for the non-translative behaviour in the more general case. In regards of scaling however, we can always assume that the solutions to \eqref{pqrdiffoporder0} are bound to the unit disc. For it is straightforward to check that the dilate $v(z):=u(\rho z)$ satisfies  \eqref{pqrdiffoporder0equation} in $\D$ whenever $u$ satisfies \eqref{pqrdiffoporder0equation} in $B_{\rho}$, excepting a multiple of $\rho^2>0$ in each parameter. Thus, and so far as the representations in \eqref{pqrorderzeroseriesintro} are concerned, these may then be recovered for any open ball $B_{\rho}$ centred at zero of radius $\rho > 0$ from those restricted to the unit disc. 

Against this background, we may then take the generalised harmonic function $u$ to be a $C^2(\D)$ function on the unit disc $\D$ that satisfies
\begin{equation}
\label{pqrdiffoporder0equationunitdisc}
M_{s,t,r}u=0\quad \text{in}\ \D.
\end{equation}  
The symmetry of the situation invokes the induced action $e^{i\theta} \in \T$ on a function $u$ on $\D$ in the form of  
\begin{equation*}
R_{e^{i\theta}} u(z) = u(e^{i\theta}z), \quad z \in \D.
\end{equation*} 
We can then ask for the possible decompositions in terms compatible with
\begin{equation}
\label{homogeneity}
R_{e^{i\theta}} u(z) = e^{im\theta}u(z), \quad z \in \D,
\end{equation}    
for $m \in \Z$. It is easy to check that the last criteria is fulfilled by the functions of form
\begin{equation}\label{mthhomogeneouspart}
u_m(z)=\frac{1}{2\pi}\int_{\mathbb{T}}e^{-im\theta}R_{e^{i\theta}}u(z) d\theta, \quad z \in \D,
\end{equation}
where $u$ is a suitably smooth function. The latter are usually referred to as the homo\-geneous parts of $u$, or the $m$:th homogeneous part in the case of a given $m \in \Z$. The name stems from their fulfilment of \eqref{homogeneity}, in which case they are also said to be of weight $m$, or homogeneous of order $m$ with respect to rotations. A generalised harmonic function is said to be decomposed or represented in such terms if
\begin{equation*}
u=\sum_{m = -\infty}^{\infty} u_m,
\end{equation*}
with convergence in $C^{\infty}(\D)$, given its usual topology.

Our primary goal is to show that a generalised harmonic function indeed has such a decomposition, and that its homogeneous parts are of the form given in \eqref{pqrorderzerobasicfunctionsintro}. Doing so involves an analysis of how \eqref{pqrdiffoporder0} acts on the homogeneous parts of such a function, which may be taken in the form of
\begin{equation}\label{homogeneousf}
u(z)=f(\lvert z\rvert^2) z^m,\quad z\in\D \setminus \{ 0 \},
\end{equation} 
for some $f \in C^2(0,1)$ and $m \in \N$, following the discussion in the fourth section of \cite{O}. In fact,
\begin{equation}\label{LtoTrule}
z^mf(|z|^2) \mapsto z^mT_{s,t,r;m}f(|z|^2), \quad m \in \N,
\end{equation}
under such an evaluation, where $T_{s,t,r;m}$ is the ordinary differential operator 
\begin{align}\label{Tgeneraliseddiffoperator}
T_{s,t,r;m;x}=x\frac{d^2}{dx^2} +[m+1-(s+t)x]\frac{d}{dx}-r-sm.
\end{align}
The existence and uniqueness of such a sequence is guaranteed by the fact that the operators in \eqref{pqrdiffoporder0} are invariant under rotations, as noted in connection with \eqref{Daction}, while the more constructive statement and the retrieval of \eqref{Tgeneraliseddiffoperator} remains subject to computations. To this end, we introduce the multiplication operator
\begin{equation}\label{multiplicationoperator}
\mathcal{M}_{\zeta;z}=\zeta(z), \quad z \in \C,
\end{equation}
that is indexed by $\zeta(z)=z$ for $z \in \C$, and acts according to $\mathcal{M}_{\zeta}f(z)=\zeta(z)f(z)$.

\begin{lemma}\label{differentiationlemmafirstorder}
Let $u$ be of the form \eqref{homogeneousf}. Then
\begin{equation*}
\left\{ 
\begin{array}{ccc}  
\mathcal{M}_{\bar \zeta} \bar \partial u &= & \mathcal{M}_{\zeta}^m|\zeta|^2(f' \circ |\zeta|^2), \\

\mathcal{M}_{\zeta} \partial u& = & m\mathcal{M}_{\zeta}^m(f \circ |\zeta|^2) +\mathcal{M}_{\zeta}^m |\zeta|^2  (f' \circ |\zeta|^2), \\    

\end{array}\right.
\end{equation*}
with equality in $\D \setminus \{0\}$, where $\mathcal{M}_{\zeta}$ is the multiplication operator in \eqref{multiplicationoperator}. 
\end{lemma}
\begin{proof}
The result follows by straightforward differentiation. As for the first of these two equalities, we have 
\begin{equation*}
\bar z \bar \partial z^mf(|z|^2)=|z|^2z^mf'(|z|^2),
\end{equation*}
and
\begin{equation*}
z\partial z^mf(|z|^2)=zmz^{m-1}f(|z|^2)+zz^m\bar zf'(|z|^2)=mz^mf(|z|^2)+z^m|z|^2f'(|z|^2),
\end{equation*}
for $z \in \D \setminus \{0\}$.  
\end{proof}

The operator 
\begin{equation}\label{angularderivative}
A_z=\mathcal{M}_{\zeta;z}\partial_z-\mathcal{M}_{\bar \zeta;z}\bar \partial_z, \quad z \in \D,
\end{equation}
that appears next is the called the angular derivative, for reasons that become clear when expressed in its corresponding polar form. 

\begin{cor}\label{angularderivativecorollary}
Let $u$ be of the form \eqref{homogeneousf}. Then
\begin{equation}\label{angulardereq}
Au=mu,
\end{equation}
with equality in $\D \setminus \{0\}$.
\end{cor}
\begin{proof}
The function $u$ can be written as $u=\mathcal{M}_{\zeta}^m(f \circ |\zeta|^2)$ in $\D \setminus\{0\}$, which makes the conclusion evident in view of Lemma \ref{differentiationlemmafirstorder}.  
\end{proof}
 
\begin{lemma}\label{differentiationlemmasecondorder}
Let $u$ be of the form \eqref{homogeneousf}. Then
\begin{equation*}
\partial\bar\partial u = (m+1) \mathcal{M}_{\zeta}^m (f' \circ |\zeta|^2)+\mathcal{M}_{\zeta}^m |\zeta|^2 (f'' \circ |\zeta|^2),
\end{equation*}
with equality in $\D \setminus \{0\}$, where $\mathcal{M}_{\zeta}$ is the multiplication operator in \eqref{multiplicationoperator}.
\end{lemma}
\begin{proof}
Note that
\begin{equation*}
\bar z \bar \partial z^m |z|^2f'(|z|^2)=|z|^2z^{m}f'(|z|^2)+|z|^2z^m |z|^2 f''(|z|^2),
\end{equation*} 
for $z \in \D \setminus \{0\}$. It then follows from Lemma \ref{differentiationlemmafirstorder} that 
\begin{equation*}
|z|^2\partial \bar \partial z^m f(|z|^2)=m|z|^2z^mf'(|z|^2)+|z|^2z^{m}f'(|z|^2)+|z|^2z^m |z|^2 f''(|z|^2),
\end{equation*}
for $z \in \D \setminus \{0\}$. This gives the desired result, following a cancellation of terms. 
\end{proof}

\begin{prop}\label{Mpqru}
Let $u$ be of the form \eqref{homogeneousf}. Then
\begin{equation}\label{actionMu}
M_{s,t,r}u=\mathcal{M}_{\zeta}^{m}(T_{s,t,r;m}f)\circ |\zeta|^2,
\end{equation}
with equality in $\D \setminus \{0\}$, where $T_{s,t,r;m}$ is the differential operator in \eqref{Tgeneraliseddiffoperator}.
\end{prop}
\begin{proof}
We can apply Lemma \ref{differentiationlemmafirstorder} and Lemma \ref{differentiationlemmasecondorder} to the first and second order terms of $M_{s,t,r}$, respectively. A summation of terms then gives \eqref{actionMu}. 
\end{proof}

\begin{cor}\label{Mpqruconj}
Let $u$ be of the form \eqref{homogeneousf}. Then 
\begin{equation}
\label{conjugateactionmu}
M_{s,t,r}\bar u=\mathcal{M}_{\bar \zeta}^{m}(T_{t,s,r;m}f)\circ |\zeta|^2,
\end{equation}
with equality in $\D \setminus \{0\}$, where $T_{t,s,r;m}$ is the differential operator in \eqref{Tgeneraliseddiffoperator}, with the complex parameters $s$ and $t$ interchanged.
\end{cor}
\begin{proof}
In view of \eqref{actionMu}, we have
\begin{equation}
\label{conjactionu}
M_{s,t,r} v=\overline{M_{\bar t, \bar s, \bar r} u}=\mathcal{M}_{\bar \zeta}^{m}(T_{t,s,r;m}f)\circ |\zeta|^2,
\end{equation}
where $v=\bar u$ denotes the complex conjugate of the function $u$.
\end{proof}
It is worth to elaborate some on the correspondence in \eqref{actionMu} between the family of operators in \eqref{pqrdiffoporder0} and the ordinary differential operators in \eqref{Tgeneraliseddiffoperator}. As before, let $\bar \Omega$ denote the four dim\-ensional vector space over $\C$ that is obtained from the family $\Omega$ in \eqref{pqrdiffoporder0} by closing it under the vector space operations, and write $A_1$ for the univariate (Weyl-)algebra generated by $x$ and $d/dx$ over $\C$. In view of the previous, we may then define the linear maps
\begin{align}\label{LtoTmap}
& \Lambda_{m}: \bar \Omega \to A_1, \quad m \in \Z,
\end{align}
that acts according to the rule imposed by \eqref{actionMu}, such that the image of $M_{s,t,r} \in \Omega$ under this map is $T_{s,t,r,m} \in A_1$ if and only if \eqref{actionMu} holds for all $u$ of the form \eqref{homogeneousf}, where for negative integers we impose \eqref{conjugateactionmu}. And while our concern here is limited to the family of operators in \eqref{pqrdiffoporder0}, these maps may as well be taken with respect to $\bar \Omega$ through Lemma \ref{differentiationlemmafirstorder} and Lemma \ref{differentiationlemmasecondorder}, or even $\mathfrak{R}_2$, as mentioned earlier. 

For the comparatively small family of operators under consideration, it will be convenient to identify the elements in $\Omega$ with their corresponding coordinates. We shall do so, and note from equation \eqref{Tgeneraliseddiffoperator} that ${\Lambda}_m$ induces a linear map with matrix representation 
\begin{equation}\label{matrixrepresentation}
\quad \Lambda_m \sim \begin{pmatrix}
    1 & 0 & 0 & 0 \\
    m+1 & 0 & 0 & 0 \\
    0 & -1 & -1 & 0 \\
    0 & -m & 0 & -1
\end{pmatrix},
\end{equation} 
for each $m \in \N$. We will refer to the corresponding matrix by the same symbol $\Lambda_m$, and record the following. 
\begin{prop}
The matrix in \eqref{matrixrepresentation} has a zero determinant. 
\hfill $\square$
\end{prop}
The last statement says that the map $\Lambda_m$ from $\C^4$ to itself fails to be injective. For a more complete picture in regards to the problem posed, it is meaningful to investigate this further, and to identify precisely which of the operators in $\Omega$ that agree in the restriction to functions of the form in \eqref{homogeneousf}. In other words, the vectors in $\C^4$ whose difference lies in the kernel of the linear map in \eqref{matrixrepresentation}. As shown below, they are precisely those points $v,w \in \C^4$ that belong to the same class under the relation imposed by $v \sim w$ if and only if 
\begin{equation}\label{equivalenceclasses}
v-w=\mu(0,1,-1,-m),
\end{equation}
for some $\mu \in \C$. 
\begin{lemma}\label{parametrizationlemma}
Let $m \in \N$. Then the vectors $v,w \in \C^4$ are equivalent in the sense of \eqref{equivalenceclasses} if and only if $v-w \in \text{ker} \ \Lambda_m$. 
\end{lemma}
\begin{proof}
Let $v$ and $w$ be two complex vectors in $\C^4$ related according to \eqref{equivalenceclasses}. A quick calculation then shows that $\Lambda_m(v-w)=0$. Conversely, if $u$ is any complex vector in $\C^4$ such that $\Lambda_m(v-u)=0$ holds, then
\begin{equation}
\label{vanishingcondition}
\left\{ 
\begin{array}{ccc}
v_1-u_1 &= & 0, \\

(m+1)(v_1-u_1) &= & 0, \\
  
(v_2-u_2)+(v_3-u_3) &= & 0, \\

m(v_2-u_2)+(v_4-u_4) &= & 0. \\   
\end{array}\right.
\end{equation}
The orthogonal complement of the vectors $(1,0,0,0)$, $(0,1,1,0)$ and $(0,m,0,1)$ in $\C^4$ is spanned by the complex vector $(0,1,-1,-m)$ in $\C^4$, from which we conclude that $u \sim v$ in the sense of \eqref{equivalenceclasses}. 
\end{proof}
Stated differently, the kernel of the linear map in \eqref{matrixrepresentation} is the one dimensional subspace of $\C^4$ given by the line in \eqref{equivalenceclasses}. Note further that \eqref{vanishingcondition} describes the vanishing conditions for operators in $\bar \Omega$ with respect to all functions of the form in \eqref{homogeneousf} for a given $m \in \N$.   

\begin{cor}\label{classmap}
Let $m \in \N$. Then the matrix in \eqref{matrixrepresentation} is a bijective linear map from the quotient space induced by the relation in \eqref{equivalenceclasses} onto its image. 
\hfill $\square$
\end{cor}

The above discussion leads to the following more concrete way of relating to the identification of operators in $\Omega$ through \eqref{equivalenceclasses}.
\begin{lemma}\label{equivalentoperatorsets}
Let $m \in \N$. Let $M_v$ and $M_w$ be operators in $\bar \Omega$ corresponding to the vectors $v$ and $w$, respectively, under the identification of $\bar \Omega$ with $\C^4$. Then $v \sim w$ in the sense of \eqref{equivalenceclasses} if and only if  
\begin{equation}\label{equivalentoperators}
M_{v}-M_{w}=\mu A-\mu m,
\end{equation}
where $A$ is the angular derivative given in \eqref{angularderivative}.
\end{lemma}
\begin{proof}
The identification of $\bar \Omega$ with $\C^4$ is a linear correspondence between vector spaces, and the vector in \eqref{equivalenceclasses} corresponds to the operator on the right hand side of \eqref{equivalentoperators} under this identification. 
\end{proof}
It is now clear precisely which of the operators in $\Omega \subset \bar \Omega$ that agree in the restriction to functions of the form in \eqref{homogeneousf} for a given $m\in \Z$, and how the operators in $\Omega$ are mapped to the corresponding ordinary differential operators in \eqref{Tgeneraliseddiffoperator} under the map $\Lambda_m$ in \eqref{LtoTmap}. 
\begin{prop}\label{quotientmapLtoT}
Let $m \in \N$ and let $A$ denote the angular derivative in \eqref{angularderivative}. Then $\Lambda_m$ is a bijective linear map from the quotient of $\bar \Omega$ over the one-dimensional subspace that is spanned by $A-m$ onto its image. 
\end{prop}
\begin{proof}
It is straightforward to check that the given map is well defined with linearity inherited from $\Lambda_m$. We may now conclude in view of Corollary \ref{classmap} and Lemma \ref{equivalentoperatorsets}. 
\end{proof}

We shall end this section with a few clarifications or implications in respect of the latter before returning to our main purpose of describing the solutions to the equation in \eqref{pqrdiffoporder0equation}.     

\begin{cor}\label{analyticcharacterizationequivalentops}
Let $m \in \N$. Let $M_v$ and $M_w$ be operators in $\bar \Omega$ corresponding to the vectors $v$ and $w$, respectively, under the identification of $\bar \Omega$ with $\C^4$. Then $v,w \in \C^4$ satisfy \eqref{equivalenceclasses} if and only if $M_vu=M_wu$ for every $u$ of the form in \eqref{homogeneousf}.
\end{cor} 
\begin{proof}
The statement is immediate in view of Proposition \ref{quotientmapLtoT}.
\end{proof}
Our next result shows that the collection of functions of the form in \eqref{homogeneousf} is also the largest set on which two equivalent operators agree, when related in the sense of \eqref{equivalenceclasses}. We will here make use of Theorem 3.5 in \cite{OLipschitz}, which provides a converse to Corollary \ref{angularderivativecorollary}. It says that a continuously differentiable function on the punctured open unit disc that satisfies \eqref{angulardereq} is homogeneous of order $m$ with respect rotations.
\begin{prop}\label{equalityequivalentoperators}
Let $m \in \N$. Let $M_v$ and $M_w$ be operators in $\bar \Omega$ corresponding to the vectors $v$ and $w$, respectively, under the identification of $\bar \Omega$ with $\C^4$. Suppose moreover that $v,w \in \C^4$ are related as in \eqref{equivalenceclasses} and that their difference is non-zero. Let $u$ be a twice continuously differentiable function on the punctured unit disc $\D$. Then $u$ is of the form \eqref{homogeneousf} if and only if $M_v u=M_w u$.  
\end{prop}
\begin{proof}
From Lemma \ref{equivalentoperatorsets}, we have that 
\begin{equation}\label{equalityLoperatorangularderivative}
M_v u-M_w u=\mu Au-\mu m u,
\end{equation}
for some $\mu \in \C \setminus \{0\}$. An application of Corollary \ref{angularderivativecorollary} shows that $u$ is of the form in \eqref{homogeneousf} only if the right side vanishes identically in $\D \setminus \{0\}$. The converse to this statement follows from the remarks preceding this statement, applied to \eqref{equalityLoperatorangularderivative}.  
\end{proof}

\section{The Kummer and Bessel equations}
 
Let $m\in \N$. The equation $T_{s,t,r,m}y=0$ that arises from the second order differential operator in \eqref{Tgeneraliseddiffoperator} can then be written as
\begin{equation}\label{kummerequation}
xy''(x)+[m+1-(s+t)x]y'(x)-(r+sm)y(x)=0,
\end{equation}
for complex numbers $s,t,r \in \C$. For $s+t \neq 0$, this equation is known as Kummer's equation, or the confluent Hypergeometric equation, usually written
\begin{equation}\label{kummergeneralequation}
xy''(x)+(b-x)y'(x)-ay(x)=0,
\end{equation}
for complex numbers $a,b \in \C$. Under certain restrictions on the parameters, it can be shown that a solution to the latter exists in the form of a complex power series, known as the confluent hypergeometric function,
\begin{equation}\label{kummergeneralfunction}
\Phi(a,b,z)=\sum_{k=0}^{\infty}\frac{(a)_k}{(b)_k}\frac{z^k}{k!}, \quad z \in \C,
\end{equation}
for $a,b \in \C$ such that $b \in \C  \setminus \{0,-1,-2,\ldots\}$. Another notation for the entire function in \eqref{kummergeneralfunction} is $_1F_1(a,b,\cdot)$, that also addresses it as a member of the greater family of hypergeometric functions $_{p}F_q(a_1, \ldots,a_p:b_1, \ldots,b_q : \cdot)$. 

When $s+t=0$, we can write \eqref{kummerequation} in the form of
\begin{equation}\label{besselequation}
xy''(x)+(m+1)y'(x)-(r+sm) y(x)=0,
\end{equation}
for complex numbers $s,r \in \C$. This last equation is closely connected with Bessel's (modified) equation 
\begin{equation}\label{besselsequation}
x^2y''(x)+xy'(x)-(x^2+n^2)y(x)=0.
\end{equation} 
A solution to this last equation can be given in the form of
\begin{equation}\label{besselfunctions}
I_{n}(z)=\frac{z^n}{2^n}\sum_{k=0}^{\infty}\frac{1}{k! \ \Gamma(n+1+k)}\frac{z^{2k}}{2^{2k}},\quad z \in \C,
\end{equation}
which is commonly referred to as the modified Bessel function.

In the subsequent two sections, and to emphasize their particular ties with the special functions just mentioned, we shall divide between the two cases $s+t \neq 0$ and $s+t =0$. The two cases will then be considered jointly, as we suggest a more transparent approach. For more on the confluent hypergeometric equation and Bessel's equation we can refer to \cite{AAR}, chapter 4. 

\section{The case of Kummer}

We return to the equation in \eqref{kummerequation} and consider the case when $s,t,r \in \C$ are such that $s+t \neq 0$. With the use of an integrating factor, we can rewrite Kummer's confluent hypergeometric equation in \eqref{kummergeneralequation} on the Sturm-Liouville form 
\begin{equation}\label{kummergeneralsturmliouville}
\frac{d}{dx}(e^{-x}x^{b}y'(x))-ae^{-x}x^{b-1}y(x)=0.
\end{equation}  
Correspondingly, we have the following form for \eqref{kummerequation}.
\begin{lemma}\label{kummersturmliouvillelemma}
Let $m \in \N$. Then the ordinary equation in \eqref{kummerequation} can be written in the form of
\begin{equation}\label{kummersturmliouville}
\frac{d}{dx}\bigg(e^{-(s+t)x}x^{m+1}y'(x)\bigg)-(r+sm)e^{-(s+t)x} x^my(x)=0, 
\end{equation} 
for $s,t,r \in \C$.
\hfill $\square$
\end{lemma}
In relation to \eqref{kummersturmliouville}, we shall consider the confluent hypergeometric function 
\begin{equation}
\label{kummerfunction}
\Phi\bigg(\frac{r+sm}{s+t},m+1,(s+t)z\bigg)=\sum_{k=0}^{\infty}\bigg(\frac{r+sm}{s+t}\bigg)_k\frac{(s+t)^k}{(m+1)_k}\frac{z^k}{k!}, \quad z \in \C,
\end{equation}
for $m \in \N$ and $s,t,r \in \C$ such that $s+t \neq 0$. The restriction relating to the sum of the two parameters $s,t \in \C$ is a convenience rather than a necessity that allows us to express the functions satisfying \eqref{kummerequation} in terms of \eqref{kummergeneralfunction}. Indeed, we have the cancellation of factors 
\begin{equation*}
(s+t)^n\bigg(\frac{r+sm}{s+t}\bigg)_n=(r+sm)(r+sm+s+t) \ldots(r+sm+(n-1)(s+t)),
\end{equation*}
for $m,n \in \N$ and $s,t,r \in \C$ such that $s+t \neq 0$. Keeping the above in mind, there is no need to impose any added condition on the sum of the two parameters, and we shall return to a more general approach following a discussion in more familiar terms. 
\begin{prop}\label{confluentsolution}
Let $m \in \N$ and let $s,t,r \in \C$ be such that $s+t \neq 0$. Then the confluent hypergeometric function in \eqref{kummerfunction} satisfies \eqref{kummerequation} on the complex domain $\C$. 
\end{prop}
\begin{proof}
By Lemma \ref{kummersturmliouvillelemma} we can write \eqref{kummerequation} on the form \eqref{kummersturmliouville}. Comparing this equation with that given in \eqref{kummergeneralsturmliouville}, we identify $a = (r+sm)/(s+t)$ and $b = m+1$ following the substitution of variables inferred from \eqref{kummerfunction}.   
\end{proof}

The following result gives our first basic construction of solutions to the equation in \eqref{pqrdiffoporder0equation}.

\begin{prop}\label{pqrharmonicfunctionsorderzerokummerthm}
Let $m \in \N$ and let $s,t,r \in \C$ be such that $s+t \neq 0$. Then
\begin{equation}\label{pqrharmonicfunctionsorderzerokummer}
u_m(z)=\Phi\big((r+sm)/(s+t),m+1,(s+t)|z|^2\big)z^m, \quad z \in \C,
\end{equation}
satisfies \eqref{pqrdiffoporder0equation}, where $\Phi(a,b,\cdot)$ is the confluent hypergeometric function in \eqref{kummergeneralfunction}. 
\end{prop}
\begin{proof}
An application of Proposition \ref{Mpqru} gives
\begin{equation*}
M_{s,t,r}u_m(z)=z^mT_{s,t,r,m}\Phi\big((r+sm)/(s+t),m+1,(s+t)|z|^2\big), \quad z \in \C.
\end{equation*}
The statement now follows from Proposition \ref{confluentsolution}.
\end{proof}
Recall the equivalence relation on the collection of operators in $\Omega$ imposed by \eqref{equivalenceclasses}. The next result shows that any function satisfying \eqref{pqrdiffoporder0equation} of the form in \eqref{pqrharmonicfunctionsorderzerokummer} is also harmonic in relation to any operator equivalent to $M_{s,t,r}$. 
\begin{cor}\label{equivalentoperatorsonpqrharmonicfunctionskummer}
Let $m \in \N$ and let $s,t,r \in \C$ be such that $s+t \neq 0$. Let $u_m$ denote the function appearing in \eqref{pqrharmonicfunctionsorderzerokummer}. Then 
\begin{equation*}
M_{s',t',r'}u_m=0, \quad \ \textnormal{in} \ \C,
\end{equation*}
for every choice of complex numbers $s',t',r' \in \C$ satisfying 
\begin{equation}\label{vectorrelationequivalentoperatorskummer}
(s',t',r')=(s+\mu,t-\mu,r-\mu m),
\end{equation}
for some $\mu \in \C$. 
\end{cor}
\begin{proof}
The result follows directly from Corollary \ref{analyticcharacterizationequivalentops}.
\end{proof}
\begin{cor}\label{pqrharmonicfunctionsorderzerokummerconjugatethm}
Let $m \in \Z_{-}$ and let $s,t,r \in \C$ be such that $s+t \neq 0$. Then
\begin{equation}\label{pqrharmonicfunctionsorderzerokummerconjugate}
u_m(z)=\Phi\big((r+t|m|)/(t+s),|m|+1,(t+s)|z|^2\big)\bar{z}^{|m|}, \quad z \in \C,
\end{equation}
satisfies \eqref{pqrdiffoporder0equation}, where $\Phi(a,b,\cdot)$ denotes the confluent hypergeometric function in \eqref{kummergeneralfunction}.
\end{cor}
\begin{proof}
The function
\begin{equation*}
v(z)=\overline{u_m(z)}=\Phi\big((\bar r+\bar t|m|)/(\bar t+ \bar s),|m|+1,(\bar t+ \bar s)|z|^2\big)z^{|m|}, \quad z \in \C,
\end{equation*}
satisfies $M_{\bar t, \bar s, \bar r}v=0$ by Proposition \ref{pqrharmonicfunctionsorderzerokummerthm}, and $M_{s,t,r}$ is the conjugate of $M_{\bar t, \bar s, \bar r}$.
\end{proof}
A few technical results will be needed in order to fully describe the homogeneous components of a function satisfying \eqref{pqrdiffoporder0equation}, which we include here for the sake of completeness. Similar results may otherwise be found in \cite{O14}.   
\begin{prop}\label{Wronskianprop}
Let $m \in \N$ and let $s,t,r \in \C$. Let $y_1$ and $y_2$ be two solutions of the equation in \eqref{kummerequation}. Then their Wronskian has the form 
\begin{equation}\label{Wronskian}
W(y_1,y_2)(x)=\begin{vmatrix}
    y_1(x) & y_2(x) \\
    y'_1(x) & y'_2(x)
\end{vmatrix}=Ae^{(s+t)x}x^{-(m+1)},
\end{equation}
for some constant $A$.  
\end{prop}
\begin{proof}
We simply note that    
\begin{equation*}
0=xW'(x)+(m+1-(s+t)x)W(x),
\end{equation*}   
where $W(x)=W(y_1,y_2)(x)$ is the Wronskian of $y_1$ and $y_2$. 
\end{proof}

\begin{lemma}\label{growthconditionlemma}
Let $m \in \N$ and let $s,t,r \in \C$. Let $y_2 \in C^2(0,\infty)$ be a solution of the equation in \eqref{kummerequation} and assume that 
\begin{equation}\label{growthassumption}(m+1)y_2'(x)-(r+sm) y_2(x)=o(1/x^{m+1}), \quad x \to 0+. 
\end{equation}
If $y_1$ is another solution to \eqref{kummerequation} that is smooth on the non-negative real axis, then the Wronskian of $y_1$ and $y_2$ in \eqref{Wronskian} vanishes identically on the positive real axis. 
\end{lemma}
\begin{proof}
By choosing an appropriate integrating factor for \eqref{growthassumption} and taking integrals, we get that $y_2(x)=o(1/x^{m})$ as $x \to 0+$ for $m > 0$ and $y_2(x)=o(\textnormal{log}(1/x))$ as $x \to 0+$ for $m=0$. Consequently,
\begin{align*}
W(y_1,y_2)(x) &=y_1(x)y_2'(x)-y_1'(x)y_2(x) \\ &=(y_2'(x)-\alpha y_2(x))y_1(x)-y_2(x)(y_1'(x)-\alpha y_1(x)) \\ &=o(1/x^{m+1}),
\end{align*} 
as $x \to 0+$, where $\alpha=(r+sm)/(m+1)$. Comparing this with Proposition \ref{Wronskianprop}, we conclude that the constant term in the expression for the Wronskian in \eqref{Wronskian} must be zero. In other words, the Wronskian $W(y_1,y_2)$ must vanish identically on the positive real axis.   
\end{proof}

\begin{prop}\label{uniquenesssolutions}
Let $m \in \N$ and let $s,t,r \in \C$ be such that $s+t\neq0$. Let the function $y \in C^2(0,\infty)$ be a solution to the equation in \eqref{kummerequation} that satisfies the growth assumption in \eqref{growthassumption}. Then $y$ is a constant multiple of the confluent hypergeometric function in \eqref{kummerfunction}. 
\end{prop}
\begin{proof}
Proposition \ref{confluentsolution} tells us that the confluent hypergeometric function in \eqref{kummerfunction} is a solution to \eqref{kummerequation}, by which we conclude following an application of Lemma \ref{growthconditionlemma}.  
\end{proof}
We note here that the last few arguments apply equally well to the more general form of Kummer's equation in \eqref{kummergeneralequation} and the confluent hypergeometric function in \eqref{kummergeneralfunction}. 

The homogeneous parts of a function $u \in C^2(\D)$ satisfying \eqref{pqrdiffoporder0equation} for $s+t \neq 0$ may now be described in full.  

\begin{thm}\label{characterizationpqrharmonicthm}
Let $m \in \N$ and let $s,t,r \in \C$ be such that $s+t\neq 0$. Let $u$ be a twice continuously differentiable on $\D$ that is homogeneous of order $m$ with respect to rotations. Then $u$ satisfies \eqref{pqrdiffoporder0equation} if and only if 
\begin{equation}\label{characterizationpqrharmonicorderzero}
u(z)=K \Phi\bigg(\frac{r+sm}{s+t},m+1,(s+t)|z|^2\bigg)z^m, \quad z \in \D,
\end{equation}
for some complex constant $K \in \C$. 
\end{thm}
\begin{proof}
Since $u$ is homogeneous of order $m \in \N$ with respect to rotations, we can write $u(z)=z^mf(|z|^2)$ for $z \in \D$ and some $f \in C^2(0,1)$. As $u \in C^2(\D)$ is bounded at the origin, we get that $f(x)=O(1/x^{m/2})$ as $x \to 0+$. By an application of Lemma \ref{differentiationlemmafirstorder}, we also have that $f'(x)=O(1/x^{(m+1)/2})$ as $x \to 0+$. The function $f$ satisfies the equation in \eqref{kummerequation} by Proposition \ref{Mpqru}, and Prop\-osition \ref{uniquenesssolutions} then implies that $u$ is of the given form. The converse of this statement follows from Proposition \ref{pqrharmonicfunctionsorderzerokummerthm}.     
\end{proof} 
 
\begin{cor}\label{equivalentoperatorspqrharmoniccharacterizationkummer}
Let $m \in \N$ and let $s,t,r \in \C$ be such that $s+t \neq 0$. If $s',t',r' \in \C$ are complex numbers satisfying \eqref{vectorrelationequivalentoperatorskummer} for some $\mu \in \C \setminus \{0\}$ and $u \in C^2(\D)$ is a function satisfying \eqref{pqrdiffoporder0equation}, then $M_{s',t',r'}u=0$ in $\D$ if and only if $u$ is of the form given in \eqref{characterizationpqrharmonicorderzero}.
\end{cor}
\begin{proof}
This follows from the previous Theorem \ref{characterizationpqrharmonicthm} taken together with Corollary \ref{equivalentoperatorsonpqrharmonicfunctionskummer}. We can also refer to Corollary \ref{analyticcharacterizationequivalentops} and Proposition \ref{equalityequivalentoperators}. 
\end{proof}

\section{The case of Bessel}

We now turn to the case where $s,t,r \in \C$ are such that $s+t=0$ and attend to the equation in \eqref{besselequation} and its solutions. As we shall see in a moment, the latter are closely associated with the modified Bessel function,   
\begin{equation}\label{modbesselfcommonbesself}
I_m(z)=i^{-m}J_{m}(iz), \quad z \in \C,
\end{equation}
where $J_{m}$ is the common Bessel function of the first kind of order $m \in \N$. 

The remark following the confluent hypergeometric function in \eqref{kummerfunction} motivates a closer look at
\begin{equation}\label{Besselpowerseries}
\Theta\big(m+1,z)=\sum_{k=0}^{\infty}\frac{1}{(m+1)_k}\frac{z^k}{k!}, \quad z \in \C,
\end{equation}
for $m \in \N$.
It is straightforward to check that the complex power series in \eqref{Besselpowerseries} is entire, or analytic on the complex domain $\C$. 
We can also express the complex power series in \eqref{Besselpowerseries} in terms of the modified Bessel function \eqref{modbesselfcommonbesself} as
\begin{equation*}
\Theta\big(m+1,z)=\Gamma(m+1)(\sqrt{z})^{-m}I_{m}(2\sqrt{z}),
\end{equation*}
for $m \in \N$, with the understanding that this function is everywhere defined. 

The next few results show that the functions in \eqref{Besselpowerseries} are of the right form for \eqref{besselequation}.  
\begin{prop}\label{propderivativebessel}
Let $m \in \N$ and let $\lambda\in \C$. Then 
\begin{equation}\label{derivativebessel}
\frac{d}{dz}\Theta(m+1,\lambda z)=\frac{\lambda}{m+1}\Theta(m+2,\lambda z), \quad z \in \C.
\end{equation}
Consequently, 
\begin{equation}\label{derivativebesselordern}
\frac{d^n}{dz^n}\Theta(m+1,\lambda z)=\frac{\lambda^n}{(m+1)_n}\Theta(m+1+n,\lambda z), \quad z \in \C.
\end{equation}
\end{prop}
\begin{proof}
Differentiating term by term we obtain
\begin{align*}
\frac{d}{dz}\Theta(m+1,\lambda z)& =\sum_{k=1}^{\infty}\frac{1}{(m+1)_k}\frac{k\lambda^k}{k!}z^{k-1}= \frac{\lambda}{m+1}\sum_{k=0}^{\infty}\frac{1}{(m+2)_k}\frac{\lambda^k}{k!}z^{k},
\end{align*}
for $z \in \C$.
\end{proof}
\begin{lemma}\label{derivativeidentitybessellemma}
Let $m \in \N$ and let $\lambda \in \C$. Then 
\begin{equation}\label{derivativeidentitybessel}
\frac{d}{dz}\bigg(z^{m+1}\frac{d}{dz}\Theta(m+1,\lambda z)\bigg)=\lambda z^{m}\Theta(m+1,\lambda z), \quad z \in \C.
\end{equation}
\end{lemma}
\begin{proof}
By \eqref{derivativebessel} of Proposition \ref{propderivativebessel}, 
\begin{equation*}
\frac{d}{dz}\bigg(z^{m+1}\frac{d}{dz}\Theta(m+1,\lambda z)\bigg)=\frac{\lambda}{m+1}\frac{d}{dz}\bigg(z^{m+1}\Theta(m+2,\lambda z)\bigg), \quad z \in \C. 
\end{equation*}
Note that  
\begin{equation*}
\frac{m+1+n}{(m+2)_n}=\frac{m+1}{(m+1)_n},
\end{equation*}
for $m,n \in \N$. This shows that
\begin{equation*}
\sum_{k=0}^{\infty}\frac{m+1+k}{(m+2)_k}\frac{\lambda^k}{k!}z^{k}=(m+1)\sum_{k=0}^{\infty}\frac{1}{(m+1)_k}\frac{\lambda^k}{k!}z^{k}, \quad z \in \C.
\end{equation*}
From this last line, we get that 
\begin{equation*}
\frac{d}{dz}\bigg(z^{m+1}\Theta(m+2,\lambda z)\bigg)=(m+1)z^m\sum_{k=0}^{\infty}\frac{1}{(m+1)_k}\frac{\lambda^k}{k!}z^{k}, \quad z \in \C,
\end{equation*}
which gives the identity in \eqref{derivativeidentitybessel}. 
\end{proof}
The identity in \eqref{derivativeidentitybessel} says that the composite of the function $\Phi(m+1,\cdot)$ with multiplication by complex numbers satisfies the differential equation
\begin{equation*}
z\frac{d^2}{dz^2}\Phi(m+1,\lambda z)+(m+1)\frac{d}{dz}\Phi(m+1,\lambda z)-\lambda\Phi(m+1,\lambda z)=0, \quad z \in \C. 
\end{equation*}

By letting $s+t=0$ in \eqref{kummersturmliouville} and comparing the resulting expression with the identity in \eqref{derivativeidentitybessel}, we obtain a solution for \eqref{besselequation} in the form of 
\begin{equation}\label{Besselpowerseriesprmform}
\Theta\big(m+1,(r+sm)z)=\sum_{k=0}^{\infty}\frac{(r+sm)^k}{(m+1)_k}\frac{z^k}{k!}, \quad z \in \C,
\end{equation}
for $s,r \in \C$ and $m \in \N$. We may treat \eqref{Besselpowerseriesprmform} as a particular instance of \eqref{kummerfunction} by the remarks that follow this last equation, and the results of the last section carry over to the present case. We thus settle with a shorter summary and refer to the last section for details. 
\begin{prop}\label{solutionpq0equation}
Let $m \in \N$ and let $s,r \in \C$. Then the complex power series in \eqref{Besselpowerseriesprmform} satisfies \eqref{besselequation} on the complex domain $\C$.
\end{prop}
\begin{proof}
Let $\lambda=r+sm \in \C$ and set
\begin{equation*}
g(z)=\Theta\big(m+1,\lambda z), \quad z \in \C.
\end{equation*}
By the identity in \eqref{derivativeidentitybessel} we have that 
\begin{equation*}\label{solutionbesseltype}
\frac{d}{dz}\bigg(z^{m+1}\frac{d}{dz}g(z)\bigg)=\lambda z^mg(z), \quad z \in \C.
\end{equation*}
By letting $s+t=0$ in \eqref{kummerequation} and considering the resulting equation \eqref{besselequation} in the form of \eqref{kummersturmliouville}, we obtain
\begin{equation*}\label{solutiongeneralbesseltype}
\frac{d}{dx}\bigg(x^{m+1}\frac{d}{dx}y(x)\bigg)=\lambda x^my(x).
\end{equation*}
A comparison between the last two expressions shows that the power series in \eqref{Besselpowerseriesprmform} is a solution to \eqref{besselequation}.
\end{proof}
\begin{prop}\label{pqrharmonicfunctionsorderzerobesselthm}
Let $m \in \N$ and let $s,t,r \in \C$ be such that $s+t=0$. Then 
\begin{equation}\label{pqrharmonicfunctionsorderzerobessel}
u_m(z)=\Theta\big(m+1,(r+sm)|z|^2)z^m, \quad z \in \C,
\end{equation}
satisfies the equation in \eqref{pqrdiffoporder0equation}, where $\Theta(m+1,\cdot)$ denotes the complex power series in \eqref{Besselpowerseries}.
\end{prop}
\begin{proof}
An application of Proposition \ref{Mpqru} gives that 
\begin{equation*}
M_{s,t,r}u_m(z)=z^mT_{s,t,r,m}\Theta\big(m+1,(r+sm)|z|^2), \quad z \in \C.
\end{equation*}
The statement now follows as a consequence of Proposition \ref{solutionpq0equation}.
\end{proof}
\begin{cor}\label{pqrharmonicfunctionsorderzerobesselconjthm}
Let $m \in \Z_-$ and let $s,t,r \in \C$ be such that $s+t=0$. Then
\begin{equation}\label{pqrharmonicfunctionsorderzerobesselconjugate}
u_m(z)=\Theta\big(|m|+1,(r+t|m|)|z|^2\big) \bar{z}^{|m|}, \quad z \in \C,
\end{equation}
satisfies \eqref{pqrdiffoporder0equation}, where $\Theta(m+1,\cdot)$ is the complex power series in \eqref{Besselpowerseries}.
\end{cor}
\begin{proof}
The result follows by taking conjugates, as in Corollary \ref{pqrharmonicfunctionsorderzerokummerconjugatethm}.
\end{proof}
\begin{thm}\label{characterizationpqrharmonicorderzerosumpq0thm}
Let $m \in \N$ and let $s,t,r \in \C$ be such that $s+t = 0$. Let $u$ be a twice continuously differentiable on $\D$ that is homogeneous of order $m$ with respect to rotations. Then $u$ satisfies \eqref{pqrdiffoporder0equation} if and only if 
\begin{equation}
\label{characterizationpqrharmonicorderzerobessel}
u(z)=K\Theta\big(m+1,(r+sm) |z|^2 \big)z^m, \quad z \in \D,
\end{equation}
for some complex constant $K \in \C$. 
\end{thm}
\begin{proof}
Lemma \ref{growthconditionlemma} and Proposition \ref{solutionpq0equation} show that any function $f \in C^2(0,1)$ that satisfies \eqref{besselequation} and the criteria in \eqref{growthassumption} is a complex multiple of the function in \eqref{Besselpowerseriesprmform}. We conclude with a similar line of argument to that given in the proof of Theorem \ref{characterizationpqrharmonicthm}.
\end{proof} 
 
The two corollaries \ref{equivalentoperatorsonpqrharmonicfunctionskummer} and \ref{equivalentoperatorspqrharmoniccharacterizationkummer} from the last section carry over more or less verbatim,  following a slight change of wording.    
 
\section{A synthesis of solutions}

The last two sections showed that solutions to \eqref{pqrdiffoporder0equation} could be constructed from the functions in \eqref{homogeneousf} that were homogeneous with respect to rotations. They were divided according to the sum of the two parameters $s,t \in \C$, emphasizing the relation to the confluent hypergeometric function in the case of $s+t \neq 0$, and the Bessel function in the case of $s+t = 0$. We also mentioned earlier in the remark following \eqref{kummerfunction} that such a distinction is unnecessary, and we shall no longer deal with the two separately. To suggest a more transparent approach, we recall the notation 
\begin{equation}\label{Pochhammersymbolmodified}
( a,\lambda )_{n}=a(a+\lambda)(a+2\lambda) \ldots(a+n\lambda-\lambda),
\end{equation}
for $n \in \N$ and $a,\lambda \in\C$, where we define $( a,\lambda )_{0}=1$. In this use of language, we can express the Pochhammer symbol as $(a)_n=( a,1 )_n$ for $n \in \N$, and introduce the complex power series    
\begin{equation}\label{pqrharmonicorderzerobasicfunction}
\mathcal{P}(r+sm,s+t|m+1;z)=\sum_{k=0}^{\infty}\frac{( r+sm,s+t )_{k}}{( m+1)_k}\frac{z^k}{k!}, \quad z \in \C,
\end{equation}  
for $m \in \N$ and $s,t,r \in \C$. It is straightforward to check that the function in \eqref{pqrharmonicorderzerobasicfunction} is entire, or analytic on the complex domain $\C$. This last expression can be viewed as an outcrop of the more general
\begin{equation}\label{kummermodifiedfunction}
\sum_{k=0}^{\infty}\frac{( a,b )_{k}}{( c,d )_k}\frac{z^k}{k!}, \quad z \in \C,
\end{equation}
for $a,b,c,d \in \C$ such that $c \neq -nd$ for $n \in \N$. As \eqref{pqrharmonicorderzerobasicfunction} is expressible as one of either of the two forms discussed previously, many of its relevant properties have already been established. We will give one of concern in the more general setting, and a formula for its derivative. 
\begin{prop}\label{derivativeproppqrharmonicorderzerobasicfunction}
Let $a,b,c,d \in \C$ be such that $c \neq -nd$ for $n \in \N$, and let $G(a,b|c,d;\cdot)$ denote the complex power series in \eqref{kummermodifiedfunction}. Then
\begin{equation}\label{derivativepqrharmonicorderzerobasicfunction}
G'(a,b|c,d;z)=\frac{a}{c}G(a+b,b|c+d,d;z), \quad z \in \C. 
\end{equation}
Consequently, 
\begin{equation}\label{nderivativepqrharmonicorderzerobasicfunction}
G^{(n)}(a,b|c,d;z)=\frac{( a,b )_n}{( c,d )_n}G(a+nb,b|c+nd,d;z), \quad z \in \C. 
\end{equation}
\end{prop}
\begin{proof}
Taking derivatives, we see that
\begin{equation*}
G'(a,b|c,d;z)=\sum_{k=1}^{\infty}\frac{( a,b )_k}{( c,d )_k}\frac{k}{k!}z^{k-1}=\sum_{k=0}^{\infty}\frac{( a,b )_{k+1}}{( c,d )_{k+1}}\frac{z^k}{k!}, \quad z \in \C.  
\end{equation*}
Since
\begin{equation*}
\frac{( a,b )_{n+1}}{( c,d )_{n+1}}=\frac{a(a+b)(a+b+b) \ldots(a+b+(n-1)b)}{c(c+d)(c+d+d) \ldots(c+d+(n-1)d)}=\frac{a}{c}\frac{( a+b,b )_{n}}{( c+d,d )_{n}},
\end{equation*}
for $n \in \N$, we obtain
\begin{equation*}
\sum_{k=0}^{\infty}\frac{( a,b )_{k+1}}{( c,d )_{k+1}}\frac{z^k}{k!}=\frac{a}{c}\sum_{k=0}^{\infty}\frac{( a+b,b )_{k}}{( c+d,d )_{k}}\frac{z^k}{k!}, \quad z \in \C. 
\end{equation*}
\end{proof}

Note that $\mathcal{P}(r+sm,s+t|m+1;\cdot)=G(r+sm,s+t|m+1,1;\cdot)$ in the notation of this last proposition, where in the former of these two the last parameter has been suppressed. We have done so to ease readability, and will make little use of the more general expression in \eqref{kummermodifiedfunction} moving forward.

By the discussion following \eqref{kummerfunction}, the complex power series in \eqref{pqrharmonicorderzerobasicfunction} can be given in such terms when $s+t \neq 0$, and in terms of the power series in \eqref{Besselpowerseriesprmform} when $s+t=0$. Thus, the results of the previous two sections carry over to \eqref{pqrharmonicorderzerobasicfunction} unabridged, and can be translated into the current use of terminology as follows.
\begin{prop}\label{pqrharmonicorderzerobasicfunctionpqnotzeroprop}
Let $m \in \N$ and let $s,t,r \in \C$ be such that $s+t \neq 0$. Then 
\begin{equation}\label{pqrharmonicorderzerobasicfunctionpqnotzero}
\mathcal{P}(r+sm,s+t|m+1;z)=\Phi\bigg(\frac{r+sm}{s+t},m+1,(s+t)z\bigg), \quad z \in \C, 
\end{equation}
where $\Phi(a,b,\cdot)$ denotes the confluent hypergeometric function in \eqref{kummergeneralfunction}.
\end{prop}
\begin{proof}
Considering the confluent hypergeometric function in \eqref{kummerfunction} and the factors appearing in each of the expansion terms, we see from the follow up remark of this expression that 
\begin{equation*}
(s+t)^n\bigg(\frac{r+sm}{s+t}\bigg)_n=( r+sm,s+t )_n,
\end{equation*}
for $n \in \N$. If we also write $(m+1)_n=( m+1,1 )_n$ for the Pochhammer symbols appearing in the denominator of each term in \eqref{kummerfunction}, we see that the two expressions agree and that equality holds in \eqref{pqrharmonicorderzerobasicfunctionpqnotzero}. 
\end{proof}
\begin{prop}\label{pqrharmonicorderzerobasicfunctionpqzeroprop}
Let $m \in \N$ and let $s,t,r \in \C$ be such that $s+t = 0$. Then 
\begin{equation}\label{pqrharmonicorderzerobasicfunctionpqzero}
\mathcal{P}(r+sm,s+t|m+1;z)=\Theta\big(m+1,(r+sm)z\big), \quad z \in \C, 
\end{equation}
where the function on the right of this equality denotes the complex power series given in \eqref{Besselpowerseriesprmform}.
\end{prop}
\begin{proof}
We simply note that   
\begin{equation*}
( r+sm,s+t )_n=(r+sm)^n,
\end{equation*}
when $s+t=0$, and that $(m+1)_n=( m+1,1 )_n$ for $n \in \N$.
\end{proof}

\section{Asymptotics}

A subset $\mathcal{F}$ of entire functions is normal if every sequence of functions in $\mathcal{F}$ has a subsequence which converges in the space of entire functions. Convergence here means convergence with respect to each of the semi-norms  
\begin{equation}\label{seminnormsnormalconvergence}
||f||_K=\max_{z \in K} |f(z)|,
\end{equation}
for $K \subset \C$ compact. We also recall from the theory of analytic functions that a subset $\mathcal{F}$ is normal if its uniformly bounded on compact subsets of $\C$, for which we can refer to \cite[Section 5.5]{LA}.

\begin{lemma}\label{normalfamily}
Let $s,t,r \in \C$. Then
\begin{equation}
|\mathcal{P}(r+sm,s+t|m+1;z)| \leq \textnormal{exp}\big((|r|+|s|+|s+t|)|z|\big), \quad z \in \C,
\end{equation}
for $m \in \N$, where $exp$ denotes the exponential function. The sequence of entire functions in \eqref{pqrharmonicorderzerobasicfunction} is therefore normal. 
\end{lemma}
\begin{proof}
By the triangle inequality, we have that 
\begin{equation*}
\bigg|\frac{r+sm+n(s+t)}{m+1+n}\bigg| \leq |r|+|s|+|s+t|, 
\end{equation*}
for $m,n \in \N$. As such, 
\begin{equation*}
\bigg|\frac{( r+sm,s+t )_n}{( m+1, 1)_n}\bigg| = \prod_{k=0}^{n-1}\bigg|\frac{r+sm+k(s+t)}{m+1+k}\bigg| \leq (|r|+|s|+|s+t|)^n,
\end{equation*}
for $m,n \in \N$. Applying the above to each term in \eqref{pqrharmonicorderzerobasicfunction} gives an estimate 
\begin{equation*}
|\mathcal{P}(r+sm,s+t|m+1;z)| \leq \sum_{k=0}^{\infty}\frac{|z|^k}{k!} (|r|+|s|+|s+t|)^k, \quad z \in \C,
\end{equation*}
for $m \in \N$.
\end{proof}
\begin{thm}\label{pqrharmonicorderzerolimit}
Let $s,t,r \in \C$. Then
\begin{equation}
\lim\limits_{m \to \infty} \mathcal{P}(r+sm,s+t|m+1;z)=e^{sz},
\end{equation}
with normal convergence in the space of entire functions on $\C$. 
\end{thm}
\begin{proof}
Let
\begin{equation}
\kappa_{m,n}=\frac{( r+sm,s+t )_n}{( m+1,1 )_{n}}.
\end{equation}
From Proposition \ref{derivativeproppqrharmonicorderzerobasicfunction} and formula \eqref{nderivativepqrharmonicorderzerobasicfunction} we have that 
\begin{equation*}
\mathcal{P}^{(n)}(r+sm,s+t|m+1;z)=\kappa_{m,n} \mathcal{P}(r+sm+n(s+t),s+t|m+1+n;z),
\end{equation*}
for $z \in \C$ and $m,n \in \N$. Evaluation at the origin gives 
\begin{equation}\label{pqrharmonicorderzerobasicfunctionatzero}
\mathcal{P}^{(n)}(r+sm,s+t|m+1;0)=\kappa_{m,n},
\end{equation}
for $m,n \in \N$. Note further that 
\begin{equation*}
\frac{sm+r+n(s+t)}{m+1+n}=\frac{s+r/m+n(s+t)/m}{1+(1+n)/m} \rightarrow s, 
\end{equation*}
as $m \rightarrow \infty$ for $n \in \N$. In view of \eqref{pqrharmonicorderzerobasicfunctionatzero}, we then obtain
$$\mathcal{P}^{(n)}(r+sm,s+t|m+1;0) \rightarrow s^n,$$as $m \rightarrow \infty$ for $n \in \N$. If we set $f(z)=e^{sz}$ for the exponential function in powers of $s$ for $z \in \C$, then $f^{(n)}(0)=s^n$ by the familiar formula for its derivative. A standard argument in line with \cite[Theorem 2.6]{kopqseries} then lets us conclude that the sequence of functions $\mathcal{P}(r+sm,s+t|m+1;\cdot)$ converges to $f$ in the space of entire functions on $\C$ as $m \to \infty$. 
\end{proof}
As an example, we give the case of Helmholtz.   
\begin{prop}\label{example1-10harmonic}
Let $u_{m}$ be a twice continuously differentiable function on $\C$ that is homogeneous of order $m \in \N$ with respect to rotations. Then $u_m$ satisfies 
\begin{equation*}
\partial \bar \partial u_{m}(z)=ru_{m}(z), \quad z \in \C,
\end{equation*}
if and only if $u_{m}$ is a complex multiple of the function
\begin{equation*}\label{-11rharmonicorderzero}
z^m\mathcal{P}(r,0|m+1;|z|^2)=z^m\sum_{k=0}^{\infty}\frac{r^k}{(m+1)_k}\frac{|z|^{2k}}{k!}, \quad z \in \C.
\end{equation*}
In the limit, we have
\begin{equation*}
\lim\limits_{m \to \infty} \mathcal{P}(r,0|m+1;z)=1,
\end{equation*}
with normal convergence in the space of entire functions on $\C$. 
\end{prop}

\section{Generalised harmonic functions and their series representations}

We will now show that the solutions to \eqref{pqrdiffoporder0equation} are smooth and can be rep\-resented as sums in terms of the basic constructions that we saw earlier, convergent on any compact set about the zero point. The given premises are thus generalised harmonic functions on open balls $B_{\rho}$ of radius $\rho > 0$ about the origin. Since scaling only has the effect of changing the parameters by a factor of $\rho^2$ however, we may restrict our attention to that of the unit disc $\D$, or go between the two as we please. While we find it convenient to state the main result of this article in latter terms, we see some merit in taking a more direct approach toward others.       

Recall that the space of smooth functions on a ball $B_{\rho}$ about the origin of radius $\rho > 0$ has a natural topology induced by the semi-norms
\begin{equation*}
||f||_{j,k,K}=\max_{z \in K}|\partial^j\bar \partial^kf(z)|,
\end{equation*} 
for $j,k \in \N$ and $K \subset B_{\rho}$ compact. The space is complete under this topology, and every absolutely convergent series in this space converges, where by convergence we mean convergence with respect to each of the semi-norms. 

\begin{prop}\label{limsupprop}
Let $s,t,r \in \C$. Then 
\begin{equation}\label{limsupstatement}
\limsup_{m \to \infty}\big(\max_{0\leq x \leq \rho} |\mathcal{P}^{(n)}(r+sm,s+t|m+1;x)|\big)^{1/m} \leq 1,  
\end{equation}
for $n \in \N$ and every positive real number $0 < \rho < \infty$.  
\end{prop}
\begin{proof}
The complex derivative $f \mapsto f'$ is continuous with respect to the topology induced by the semi-norms in \eqref{seminnormsnormalconvergence} and normal families of analytic functions on $\C$ are uniformly bounded on compact subsets, which allows us to conclude with \eqref{limsupstatement}, following an application of Theorem \ref{pqrharmonicorderzerolimit}. 
\end{proof}

We define the $m$-th homogeneous part of a suitably smooth function $u$ on the unit disc $\D$ by the integral expression
\begin{equation}\label{mthhomogeneouspart}
u_m(z)=\frac{1}{2\pi}\int_{\mathbb{T}}e^{-im\theta}u(e^{i\theta}z) d\theta, \quad z \in \D,
\end{equation}
for $m \in \Z$. It is straightforward to check that the $m$-th homogeneous part $u_m$ is homogeneous of order $m$ with respect to rotations and that it inherits the regularity of its source, such that $u_m$ is $k$-times continuously differentiable whenever $u$ is. Since the operators in \eqref{pqrdiffoporder0} are rotationally symmetric, it also follows that if $u$ satisfies \eqref{pqrdiffoporder0equation}, then likewise does its $m$-th homogeneous part. 

\begin{prop}\label{characterizationmthhomogeneouspart}
Let $s,t,r \in \C$ and let $m \in \N$. Let $u$ satisfy \eqref{pqrdiffoporder0equation} on $\D$ and let $u_m$ be its $m$-th homogeneous part. Then 
\begin{equation}\label{characterizationmthhompart}
u_m(z)=a_m\mathcal{P}(r+sm,s+t|m+1;|z|^2)z^m, \quad z \in \D,
\end{equation} 
for a constant $a_m \in \C$.
\end{prop}
\begin{proof}
From the preceding comments, we know that $u_m$ is homogeneous of order $m$ with respect to rotations, and that it satisfies \eqref{pqrdiffoporder0equation}. In the case of $s+t \neq 0$, we can apply Theorem \ref{characterizationpqrharmonicthm} together with Proposition \ref{pqrharmonicorderzerobasicfunctionpqnotzeroprop} for the result of this statement. Theorem \ref{characterizationpqrharmonicorderzerosumpq0thm}  taken together with Proposition \ref{pqrharmonicorderzerobasicfunctionpqzeroprop} gives the case for $s+t=0$.
\end{proof}
\begin{cor}\label{characterizationmthhomogeneouspartcor}
Let $s,t,r\in \C$ and let $m \in \Z_-$. Let $u$ satisfy \eqref{pqrdiffoporder0equation} on $\D$ and let $u_m$ be its $m$-th homogeneous part. Then 
\begin{equation}\label{characterizationmthhompartconjugate}
u_m(z)=b_m\mathcal{P}(r+t|m|,t+s\big ||m|+1;|z|^2)\bar z^{|m|}, \quad z \in \D,
\end{equation} 
for a constant $b_m \in \C$.
\end{cor}
\begin{proof}
Let $v_m$ be the conjugate of the function on the right hand side of \eqref{characterizationmthhompartconjugate}, and notice from \eqref{mthhomogeneouspart} that $\bar u_m$ is the $|m|$-th homogeneous part of the function $\bar u$. In view of Proposition \ref{characterizationmthhomogeneouspart} and the fact that $\overline{M_{s,t,r}}=M_{\bar t, \bar s, \bar r}$, we may thus write $\bar u_m=\bar b_m v_m$ for some $\bar b_m \in \C$. 
\end{proof}

In the sequel, we shall consider infinite sums of smooth functions  $\{u_m\}_{m=0}^{\infty}$ and make reference to the number $\limsup_{m \to \infty}||u_m||_{j,k,K}^{1/m}$. This number provides a classical root test for convergence, and the sum $\sum_{m=0}^{\infty}||u_m||_{j,k,K} < +\infty$ converges for $j,k \in \N$ and $K \subset B_{\rho}$ compact when this number is less than unity.

\begin{prop}\label{limsupballthm}
Let $\rho > 0$ and let $\{f_m\}_{m=0}^{\infty}$ be a sequence in $C^{\infty}[0,\rho)$ such that 
\begin{equation}
\label{limsupboundradialpart}
\limsup_{m \to \infty}\bigg(\max_{0 \leq x \leq r} |f_m^{(n)}(x)|\bigg)^{1/m}\leq \frac{1}{\sqrt{\rho}},
\end{equation}
for $n \in \N$ and $0 < r <\rho$. Let $B_{\sqrt{\rho}} \subset \C$ denote the open ball of radius $\sqrt{\rho}$ centred at the origin and consider the sequence of functions
\begin{equation}\label{umballthm}
u_m(z)=f_m(|z|^2)z^m, \quad z \in B_{\sqrt{\rho}},
\end{equation}
for $m \in \N$. Then $\limsup_{m \to \infty}||u_m||_{j,k,K}^{1/m}<1$ for every $j,k \in \N$ and compact subset $K \subset B_{\sqrt{\rho}}$.
\end{prop}
\begin{proof}
Let $j,k \in \N$ be whole numbers and let $K$ be a given compact set contained in the open ball $B_{\sqrt{\rho}}=\{z \in \C: |z| < \sqrt{\rho}\}$ of positive radius $\sqrt{\rho}$ centred at zero. Take $\omega$ to be the maximum distance $|z| < \sqrt{\rho}$ as $z$ varies over the compact set $K \subset B_{\sqrt{\rho}}$. 

It is straightforward to check that 
\begin{equation*}
\partial^j \bar \partial^ku_m(z)=z^{m+k-j}\sum_{l=0}^j\binom{j}{l}\frac{(m+k)!}{(m+k+l-j)!}|z|^{2l}f_m^{(k+l)}(|z|^2),
\end{equation*}
for $z \in B_{\sqrt{\rho}}$ and $m \geq j$. By the triangle inequality, we can then write
\begin{equation*}
||u_m||_{j,k,K} \leq \omega^{m+k-j}\bigg(\sum_{l=0}^j\binom{j}{l}\frac{(m+k)!}{(m+k+l-j)!}\omega^{2l}\bigg)\max_{k \leq n \leq j+k}||f_m^{(n)}||_{[0,\omega^2]},
\end{equation*}
for $m \geq j$, where $|| f ||_{[0,w^2]}=\sup\{|f(x)| : 0 \leq x \leq \omega^2\}$. Note that 
\begin{equation*}
\frac{(m+k)!}{(m+k+n-j)!}=(m+k)\ldots(m+k+1-j)\frac{(m+k-j)!}{(m+k+n-j)!}\leq(m+k)^j,
\end{equation*}
for $j,k,m,n \in \N$ such that $m \geq j$, and recall the expansion formula
\begin{equation*}
(1+x)^n=\sum_{l=0}^n\binom{n}{l}x^{l}.
\end{equation*}
Thus,
\begin{equation*}
||u_m||_{j,k,K} \leq \omega^{m+k-j}(1+\omega^2)^j(m+k)^j\max_{k \leq n \leq j+k}||f_m^{(n)}||_{[0,\omega^2]},
\end{equation*}
for $m \geq j$. By taking the $m$-th root on both sides and passing to the limit, we then obtain 
\begin{equation*}
\limsup_{m \to \infty}||u_m||_{j,k,K}^{1/m} \leq \frac{\omega}{\sqrt{\rho}}.
\end{equation*}
Since the number $\omega$ is strictly smaller than $\sqrt{\rho}$, we may now conclude. 
\end{proof}

\begin{lemma}\label{limitcoefficients}
Let $s,t,r\in \C$ and suppose that $u$ satisfies \eqref{pqrdiffoporder0equation} on $\D$. Let $k_m=a_m$ be the complex constant appearing in the expression for the $m$-th homogeneous part of $u$ in \eqref{characterizationmthhompart} for $m \in \N$, and let $k_m=b_m$ be the complex constant appearing in the expression for the $m$-th homogeneous part of $u$ in \eqref{characterizationmthhompartconjugate} for $m \in \Z_-$. Then 
$$\limsup_{|m| \to \infty}|k_m|^{1/|m|} \leq 1.$$   
\end{lemma}
\begin{proof}
Let $0 < \rho < 1$ and note from \eqref{mthhomogeneouspart} that 
\begin{equation*}
\max_{|z|=\rho}|u_m(z)| \leq \max_{|z|=\rho} |u(z)|,
\end{equation*}
for $m \in \Z$. It follows from Proposition \ref{characterizationmthhomogeneouspart} and \eqref{characterizationmthhompart} that 
\begin{equation*}
|k_m||\mathcal{P}(r+sm,s+t|m+1;\rho^2)|\rho^m \leq \max_{|z|=\rho}|u(z)|,
\end{equation*}
for $m \in \N$. We also know from Theorem \ref{pqrharmonicorderzerolimit} that
\begin{equation*}
\lim_{m \to \infty} \mathcal{P}(r+sm,s+t|m+1;\rho^2) = e^{s\rho^2}.
\end{equation*}
If we now take the $m$-th root of each side in the last inequality and then pass to the limit, we find that    
\begin{equation*}
\limsup_{m \to \infty}|k_m|^{1/m} \leq \frac{1}{\rho}.
\end{equation*} 
Since this is true for all $0 < \rho < 1$, we get the case for non-negative $m \in \N$. The case for $m \in \Z_-$ is similar and we conclude accordingly. 
\end{proof}
We recall from Proposition \ref{characterizationmthhomogeneouspart} and Corollary \ref{characterizationmthhomogeneouspartcor} that the $m$-th homo\-geneous parts of a generalised harmonic function satisfying \eqref{pqrdiffoporder0equation} are all smooth. 

\begin{prop}\label{roottestcriteria}
Let $s,t,r \in \C$. Let $u$ satisfy \eqref{pqrdiffoporder0equation} on $\D$ and write $u_m$ for its $m$-th homogeneous part for $m \in \Z$. Then
\begin{equation*}
\limsup_{|m| \to \infty}||u_m||_{j,k,K}^{1/|m|}  < 1,
\end{equation*}
for $j,k \in \N$ and every compact set $K \subset \D$. 
\end{prop}
\begin{proof}
Let $m \in \N$. As $u_m$ is the $m$-th homogeneous part of the function $u$, it is of the form \eqref{characterizationmthhompart} for some $a_m \in \C$.
From $\eqref{limsupstatement}$ and Lemma \ref{limitcoefficients} we see that the criteria in \eqref{limsupboundradialpart} is satisfied for $\rho=1$. Thus $\limsup_{m \to \infty}||u_m||_{j,k,K}^{1/m} < 1$ for $j,k \in \N$ and $K \subset \D$ compact. The result for $m \in \Z_-$ follows by applying the previous part to the sequence of conjugated elements $\bar{u}_m$, noting that $\bar{u}_m$ is the $|m|$-th homogeneous part of the generalised harmonic function $\bar{u}$.    
\end{proof}

The following result shows that every generalised harmonic function satisfying \eqref{pqrdiffoporder0equation} can be expanded as a sum in terms of its homogeneous parts. For this we recall that if $u$ is an $n$-times continuously differentiable function on $\D$, then
\begin{equation}\label{fejermeans}
u=\lim_{N \to \infty} \sum_{m=-N}^{N}\bigg(1-\frac{|m|}{N+1}\bigg)u_m, 
\end{equation}
in $C^{n}(\D)$, where $u_m$ is the $m$-th homogeneous part of $u$. A more elaborative description of this result, originally due to Fej\'er, can be found in \cite{Katznelson}, section I.2.

\begin{cor}\label{convergencemthhompartseries}
Let $s,t,r \in \C$. Let $u$ satisfy \eqref{pqrdiffoporder0equation} on $\D$ and write $u_m$ for its $m$-th homogeneous part for $m \in \Z$. Then the series of complex functions $\sum_{m = -\infty}^{\infty} u_m$ is absolutely convergent in the space of smooth functions on $\D$. In particular,
\begin{equation}\label{sumconvergence}
u=\sum_{m = -\infty}^{\infty} u_m,
\end{equation}
in $C^{\infty}(\D)$. 
\end{cor}
\begin{proof}
By the root test and the previous Prop\-osition \ref{roottestcriteria}, we conclude that the sum converges absolutely in $C^{\infty}(\D)$, and so converges in $C^{\infty}(\D)$. This gives the first part of the statement. As for the second part, we can write $u$ on the form \eqref{fejermeans} with convergence in $C^{2}(\D)$. A standard argument now shows that we can escape the convergence factors in \eqref{fejermeans} and conclude with \eqref{sumconvergence}. 
\end{proof}
 
Finally then, we may state the following.  

\begin{thm}\label{seriesexpansionpqrharmonicfunctionsorderzero}
Let $s,t,r \in \C$. Then $u$ satisfies \eqref{pqrdiffoporder0equation} on $\D$ if and only if it can be written as a sum 
\begin{align}\label{seriesexpansionorderzero}
u(z)&=\sum_{m=0}^{\infty}k_m\mathcal{P}(r+sm,s+t|m+1;|z|^2)z^m \\ &+\sum_{m=1}^{\infty}k_{-m}\mathcal{P}(r+tm,t+s|m+1;|z|^2)\bar z^m \notag,
\end{align}
for $z \in \D$ and some sequence $\{k_m\}_{m \in \Z}$ of complex numbers satisfying
\begin{equation}\label{coefficientscondition}
\limsup_{|m| \to \infty}|k_m|^{1/|m|} \leq 1.
\end{equation}
The sum in \eqref{seriesexpansionorderzero} is absolutely convergent in $C^{\infty}(\D)$ when $\{k_m\}_{m \in \Z}$ is a sequence satisfying \eqref{coefficientscondition}.
\end{thm}
\begin{proof}
Let $\{k_m\}_{m \in \Z}$ be a sequence of complex numbers such that \eqref{coefficientscondition} holds and set
\begin{equation}\label{mthhompartseriesexpansion}
u_m(z)=k_m\mathcal{P}(r+sm,s+t|m+1;|z|^2)z^m, \quad z \in \D,
\end{equation}
for $m \in \N$, noting that these are the solutions to \eqref{pqrdiffoporder0equation} that were set up in \eqref{characterizationmthhompart}. Similarly, we let 
\begin{equation}\label{mthhompartsconjugateseriesexpansion}
u_m(z)=k_m\mathcal{P}(r+t|m|,t+s||m|+1;|z|^2)\bar z^{|m|}, \quad z \in \D,
\end{equation}
for $m \in \Z_{-}$, and recall the functions in \eqref{characterizationmthhompartconjugate}. 

As for the first part, consider the formal expression 
\begin{equation}\label{formalsum}
u \sim \sum_{m=-\infty}^{\infty} u_m. 
\end{equation}
The same argument that we gave for Proposition \ref{roottestcriteria} shows that 
\begin{equation*}\label{roottestseriesexpansion}
\limsup_{|m| \to \infty}||u_m||_{j,k,K}^{1/|m|} < 1, 
\end{equation*} 
for $j,k \in \N$ and compact sets $K \subset \D$, with $u_m$ as in \eqref{mthhompartseriesexpansion} for $m \in \N$ or \eqref{mthhompartsconjugateseriesexpansion} in the case of $m \in \Z_-$. An application of the root test then shows that the sum \eqref{formalsum} converges in $C^{\infty}(\D)$, and allows us to write \eqref{formalsum} with equality. In the case of $s+t \neq 0$, the results of Proposition \ref{pqrharmonicfunctionsorderzerokummerthm} and Corollary \ref{pqrharmonicfunctionsorderzerokummerconjugatethm} now carry over through Proposition \ref{pqrharmonicorderzerobasicfunctionpqnotzeroprop} and show that the functions in \eqref{mthhompartseriesexpansion} and \eqref{mthhompartsconjugateseriesexpansion} satisfy \eqref{pqrdiffoporder0equation}. Proposition \ref{pqrharmonicfunctionsorderzerobesselthm} and Corollary \ref{pqrharmonicfunctionsorderzerobesselconjthm} taken together with Proposition \ref{pqrharmonicorderzerobasicfunctionpqzeroprop} give a similar conclusion in the case of $s+t=0$. 
Thus $u$ is a generalised harmonic in $C^{\infty}(\D)$, and completes the first part. 

Conversely, let $u$ satisfy \eqref{pqrdiffoporder0equation}. If we denote its $m$-th homogeneous part by $u_m$, then Proposition \ref{characterizationmthhomogeneouspart} and Corollary \ref{characterizationmthhomogeneouspartcor} tell us that the $m$-th homogeneous parts are of either of the two forms \eqref{mthhompartseriesexpansion} or \eqref{mthhompartsconjugateseriesexpansion}, depending on whether $m \in \N$ or $m \in \Z_-$. Lemma \ref{limitcoefficients} and Corollary \ref{convergencemthhompartseries} now allow us to conclude with the statement of this Theorem.  
\end{proof}

We shall continue in the vein of earlier expositions and scetch a few results concerning the expansion coefficients in \eqref{seriesexpansionorderzero}. 

\begin{prop}
Let $s,t,r \in \C$. Let $u$ be a generalised harmonic function on $\D$ with series representation \eqref{seriesexpansionorderzero} that is characterized by the sequence of coefficients $\{k_m\}_{m \in \Z}$ conditioned by \eqref{coefficientscondition}. Then 
\begin{equation}\label{integralexpressionkm}
k_m = \lim_{\rho \to 0} \frac{1}{2\pi \rho^{|m|}}\int_{\mathbb{T}}u(\rho e^{i \theta})e^{-im\theta} d\theta,
\end{equation} 
for $m \in \Z$. 
\end{prop}  
\begin{proof}
Let $m \in \N$. By expressing $u$ in the form of \eqref{seriesexpansionorderzero} and integrating termwise, we get that 
\begin{equation*}\label{integralcoefficients}
\frac{1}{2\pi}\int_{\mathbb{T}}u(\rho e^{i \theta})e^{-im\theta} d\theta=k_m\mathcal{P}(r+sm,s+t|m+1;\rho^2)\rho^m,
\end{equation*}
for $0 < \rho < 1$. Dividing through by $\rho^m$ and taking the limit gives the case for $m \in \N$. The case for negative integers is similar.  
\end{proof}

\begin{thm}\label{coefficientsdetermined}
Let $s,t,r \in \C$. Let $u$ be a generalised harmonic function on $\D$ with series representation given by \eqref{seriesexpansionorderzero} for some expansion coefficients $\{k_m\}_{m \in \Z}$ conditioned by \eqref{coefficientscondition}. Then
\begin{equation}
k_m=\partial^m u(0)/m! \quad  \text{and} \quad  k_{-m}=\bar \partial^m u(0) / m!,
\end{equation}
for $m \in \N$. 
\end{thm}
\begin{proof}
The result can be proven by considering the Taylor expansion of $u$ about the origin and inserting the resulting expression into the integral on the right side of \eqref{integralexpressionkm}, followed by a passage to the limit in $\rho$. We refer to \cite[Theorem 5.3]{kopqseries} for details. 
\end{proof}
In summary, we arrive at the following unique representation for the generalised functions under consideration. 
\begin{cor}\label{seriesexpansionorderzerocompletethm}
Let $s,t,r \in \C$. Let $u$ satisfy \eqref{pqrdiffoporder0equation} on $\D$. Then
\begin{align*}\label{seriesexpansionorderzerocomplete}
u(z)&=\sum_{m=0}^{\infty}\frac{\partial^mu(0)}{m!}\mathcal{P}(r+sm,s+t|m+1;|z|^2)z^m \\ &+\sum_{m=1}^{\infty}\frac{\bar \partial^mu(0)}{m!}\mathcal{P}(r+tm,t+s|m+1;|z|^2)\bar z^m \notag,
\end{align*}
for $z \in \D$. 
\end{cor}
\begin{proof}
The result is a straightforward consequence of the preceding Theorem \ref{seriesexpansionpqrharmonicfunctionsorderzero} and Theorem \ref{coefficientsdetermined}.
\end{proof}

We conclude with an example and return to the case of Helmholtz in Proposition \ref{example1-10harmonic}.

\begin{prop}
Let $r \in \C$ and let $u$ be a twice continuously differentiable function on $\D$ that satisfies
\begin{equation*}
\partial \bar \partial u-ru=0, \quad \textnormal{in} \ \D.
\end{equation*}
Then
\begin{equation*}
u(z)=\sum_{m=1}^{\infty}\sum_{n=0}^{\infty}\bigg(\frac{ z^{m}\partial^m u(0)+\bar z^{m} \bar \partial^m u(0) }{(m+n)!}\bigg)\frac{r^n|z|^{2n}}{n!}+\sum_{n=0}^{\infty} \frac{u(0)}{n!}\frac{r^n |z|^{2n}}{n!}, \quad z \in \D.
\end{equation*}
In particular, 
\begin{equation*}
u(z)=\sum_{m=1}^{\infty}\frac{ z^{m}\partial^m u(0)+\bar z^{m} \bar \partial^m u(0)}{m!}+u(0), \quad z \in \D,
\end{equation*}
when $r=0$. 
\end{prop} 

\section*{Acknowledgements}

The author would like to thank Anders Olofsson and Jens Wittsten for their valuable comments and suggestions in relation to this text. 

\bigskip

Research was partially supported by the Swedish Research Council grant number 2019-04878.  

\end{document}